\documentclass[11pt,a4paper]{amsart}

\usepackage[left=3cm,right=3cm,top=3cm,bottom=2cm,foot=1.7cm]{geometry}
\usepackage{amsmath,amssymb,amsthm,mathrsfs}
\usepackage{amsaddr}

\usepackage{enumerate}
\usepackage{color}
\newtheorem{theorem}{Theorem}
\newtheorem{theoremx}{Theorem}

\newtheorem{lemma}{Lemma}

\theoremstyle{definition}
\newtheorem{definition}{Definition}
\newtheorem*{remark}{Remark}

\title[Asymptotic degree distribution in preferential attachment graph models]{Asymptotic degree distribution in preferential attachment graph models with multiple type edges}

\author{\'Agnes Backhausz}
\address{Department of Probability Theory and Statistics\\Faculty of Science\\ELTE E\"otv\"os Lor\'and University, Budapest, Hungary\\and\\MTA Alfr\'ed R\'enyi Institute of Mathematics, Budapest, Hungary}
\email{agnes@math.elte.hu}

\author{Bence Rozner}
\address{Department of Probability Theory and Statistics\\Faculty of Science\\ELTE E\"otv\"os Lor\'and University, Budapest, Hungary}
\email{robsaat@caesar.elte.hu}


\date{\today}
\begin{document}
\maketitle
\thispagestyle{empty}

\begin{abstract}
	We deal with a general preferential attachment graph model with multiple type edges. The types are chosen randomly, in a way that depends on the evolution of the graph. In the $N$-type case, we define the (generalized) degree of a given vertex as $\boldsymbol{d}=(d_{1},d_{2},\dots,d_{N})$, where $d_{k}\in\mathbb{Z}_{0}^{+}$ is the number of type $k$ edges connected to it. We prove the existence of an a.s.\ asymptotic degree distribution for a general family of preferential attachment random graph models with multi-type edges. More precisely, we show that the proportion of vertices with (generalized) degree $\boldsymbol{d}$ tends to some random variable as the number of steps goes to infinity. We also provide recurrence equations for the asymptotic degree distribution. Finally, we generalize the scale-free property of random graphs to the multi-type case.
\end{abstract}

\section{Introduction}
Various types of random graphs with preferential attachment dynamics have been examined in the last decade, see e.g.\ \cite{BarabasiAlbert_Albert,Cooper_Frieze,Durrett,Frieze_Karonski,Hofstad}. The analysis of these kind of random graphs is motivated by large real networks, such as the internet and various biological and social networks, in which vertices of larger degree have more chance to be connected to new vertices. In many applications, it is natural to assign some kind of characteristics to the vertices or to the edges of the graph. For example, the strength of a connection may be represented by edge weights, or vertices can have different fitness, which has an impact on their degrees, see e.g.\ \cite{Dereich_Ortgiese,Garavaglia_Hofstad}. It may also happen that the type of a vertex or an edge is chosen from a finite set of possibilities. This leads to different phenomena as if we assign weights to the vertices or to the edges. For example, in a social network, the vertices can be considered as males or females, and the edges can be considered as family or work relationships. Another example is the network of financial systems, where the systemic risk is examined, see e.g.\ \cite{Acemoglu}. To understand these kind of financial systems it is common to use graphs where the vertices are financial institutions (e.g.\ banks), and the edges represent different types of financial instruments traded by the institutions. The risk arising from these instruments (bonds, stocks or options etc.) can be different, which must be taken into account in the calculation of the systemic risk. A way to do this is to assign types for the edges to represent the classes of these assets. To model folded RNA-molecules, David, Hagendorf and Wiese introduced a random graph in \cite{David} which grows by a process similar to the preferential attachment and there are two types of vertices.

There are some multi-type preferential attachment graph models that have been investigated in which only the vertices have types. Antunovi\'c, Mossel and R\'acz introduced a model of competition on growing networks in \cite{Racz}. In their model, when a new vertex is born, it attaches to the old vertices by preferential attachment, and selects its type based on the number of its initial neighbours of each type. Their main interest is the question of coexistence, i.e.\ the probability that one of the types dies out asymptotically. Abdullah, Bode and Fountoulakis present a model in \cite{Abdullah}, but they use a different rule for choosing the types. At each step, a new vertex is born, it polls some of the old vertices and takes the majority type. A multi-type preferential attachment model was introduced by Rosengren in \cite{Rosengren} which has similar dynamics to the model presented in \cite{Racz}. The asymptotic degree distribution is examined by using methods from the theory of multi-type branching processes.

Notice that the growing networks in the 2-type case can equivalently be viewed as a directed graph. In this case the types of the edges are orientations, more precisely, when there is a new vertex then it is attached to the graph with an edge from the new vertex to the existing ones or from the existing vertices to the new one, and this corresponds to two different types. Different directed preferential attachment models were introduced in \cite{Bolobas_Borgs_Chayes_Riordan,Wang_Resnick}. They examine a growing network in which a new vertex and a new edge is added to the graph in every step. At first, the orientation of the edge between the new and the existing vertices is decided with fixed probability. Finally the endpoint of the new edge among the existing vertices is chosen by using a preferential attachment rule. In \cite{Bolobas_Borgs_Chayes_Riordan} it is also possible that the new edge is added between existing vertices. In \cite{Bolobas_Borgs_Chayes_Riordan,Wang_Resnick} the asymptotic degree distribution is examined. In those models which are discussed in this article, we first choose the endpoint of the new vertex and then the type of the new edge is decided with probabilities depending on the structure of the graph.

In this paper we extend the preferential attachment model by assigning types to the edges. For trees, this is usually not an essential difference compared to the cases where the vertices have types, but we consider more complex networks. We assume that there is a connection between the evolution of the structure of the graph and the types of the edges. In the $N$-type case, we define the (generalized) degree of a given vertex as $\boldsymbol{d}=(d_{1},d_{2},\dots,d_{N})$, where $d_{k}$ is the number of type $k$ edges connected to it. By using martingale techniques, we prove the existence of an almost sure asymptotic degree distribution. More precisely, we show that for every $\boldsymbol{d}$, the proportion of vertices with generalized degree $\boldsymbol{d}$ tends to some random variable in certain random graph models with multiple type edges as the number of steps goes to infinity. We also provide recurrence equations for the asymptotic degree distribution. The results are verified not just for particular graph models; instead, we follow a model-free approach
and formulate sufficient conditions for the existence of asymptotic degree distribution. Then we give two applications: for a multi-type version of the Barab\'asi--Albert random graph, and for a preferential attachment model with Poisson number of edges. These examples show a new phenomenon: in the multi-type case it can happen that the asymptotic degree distribution is not deterministic, which is the case in many well-known models in the single-type case. We show that the asymptotic degree distribution in the generalized Barab\'asi--Albert random graph and in the model of independent edges also depends on the asymptotic proportion of edges of type $k$ which makes it a stochastic distribution.

The scale-free property of random graph models is a well-studied feature in the single-type case and also very important in different applications, see e.g.\ \cite{Hofstad}. We generalize this property in the multi-type case, and calculate the generalized characteristic exponent in the multi-type Barab\'asi--Albert random graph and in the model of independent edges.

{\bf Outline.} In Section \ref{sec_NotationsAssumption}, we list the notation and the assumptions on the general model. In Section \ref{sec_MainResults}, we formulate the main results, and we introduce two random graphs, which are special cases of the general model: the generalized Barab\'asi--Albert random graph and the model of independent edges. In Section 4, the proofs of the main theorems are given. Finally, we generalize the scale-free property of random graphs to the multi-type case in Section 5.

\section{Notation and assumptions}
\label{sec_NotationsAssumption}
\subsection{Notation}
Let $\left(G_{n}\right)_{n=0}^{\infty}$ be a sequence of finite random graphs. The vertex set and the edge set of $G_{n}$ are denoted by $V_{n}$ and $E_{n}$, respectively. In the sequel, $N$ will be fixed, this is the number of possible types of edges. For every $k\in[N]=\{1,\dots,N\}$ let $E_{n}^{(k)}$ denote the set of edges with type $k$ in $G_{n}$. For every $n$ we have $E_{n}=\bigcup_{k\in[N]}E_{n}^{(k)}$, and we assume that $E_{n}^{(k)}\subseteq{}E_{n+1}^{(k)}$ for every $k\in[N]$.

\begin{definition}
	For every $n$ the generalized degree of a vertex $v\in{}V_{n}$ in the $n$th step is $\mathbf{deg}_{n}(v)=\left(\mathrm{deg}_{n}^{(k)}(v)\right)_{k=1}^{N}$, where $\mathrm{deg}_{n}^{(k)}(v)$ is the number of edges of type $k$ connected to $v$ in $G_{n}$.
\end{definition}

The initial configuration is denoted by $G_{0}=(V_{0},E_{0})$, where $V_{0}=\{u_{1},u_{2},\dots,u_{s}\}$ ($s\geq{}1$). We allow multiple edges, but loops are forbidden. We assume that for every $k\in[N]$ we have $\big|E_{0}^{(k)}\big|>0$.

For every $n$, in the $n$th step,
\begin{enumerate}
	\item{}a new vertex $v_{n}$ is born, thus $V_{n}=V_{0}\cup\{v_{1},v_{2},\dots,v_{n}\}$;
	\item{}the new vertex $v_{n}$ attaches with a few edges to some of the old vertices, so every element of the edge set $E_{n}\setminus{}E_{n-1}$ is connected to $v_{n}$;
	\item{}every new edge gets a type randomly. For example, we can consider the following case: for every $n$, in the $n$th step, any edge between the new vertex $v_{n}$ and an existing vertex $v\in{}V_{n-1}$ will be assigned to type $k$ with probabilities proportional to $\textrm{deg}_{n-1}^{(k)}(v)$ for every $k\in[N]$.
\end{enumerate}

For every $\boldsymbol{d}\in(\mathbb{Z}_{0}^{+})^{N}=\left\{(x_{1},\dots,x_{N})\in\mathbb{Z}^{N}:x_{k}\geq{}0\textrm{ for every }k\in[N]\right\}$ we define
\begin{align*}
	X_{n}(\boldsymbol{d})=\left|\{v\in{}V_{n}:\mathbf{deg}_{n}(v)=\boldsymbol{d}\}\right|,
\end{align*}
this is the number of vertices in $G_{n}$ with generalized degree $\boldsymbol{d}$. Finally, for every $n\geq{}1$ let $\mathcal{F}_{n}$ denote the $\sigma$-algebra generated by the first $n$ graphs, and let $\mathcal{F}_{0}$ be the trivial $\sigma$-algebra, thus $\mathcal{F}=(\mathcal{F}_{n})_{n=0}^{\infty}$ is a filtration.

Throughout the paper $\boldsymbol{e}_{k}$ will be the $k$th unit vector in $(\mathbb{Z}_{0}^{+})^{N}$.

\subsection{Assumptions}\ Now we list the assumptions we are going to use throughout the paper.

\textbf{Assumption 1.}\ For every $n\geq{}1$ we assume that in the $n$th step, conditionally with respect to $\mathcal{F}_{n-1}$, the conditional distribution of the number of new edges of type $k$ connected to an existing vertex $v\in{}V_{n-1}$ depends only on $\mathrm{deg}_{n-1}^{(k)}(v)$ for every $k\in[N]$. By using this assumption, for every $\boldsymbol{d},\boldsymbol{\gamma}\in(\mathbb{Z}_{0}^{+})^{N}$ let $p_{\boldsymbol{d}}^{(n)}(\boldsymbol{\gamma})$ denote the conditional probability that, with respect to $\mathcal{F}_{n-1}$, a vertex with generalized degree $\boldsymbol{d}$ gets exactly $\gamma_{k}$ edges of type $k$ in the $n$th step.

\textbf{Assumption 2.}\ For every $\boldsymbol{d}\in(\mathbb{Z}_{0}^{+})^{N}$, there exists $\delta>0$ and $C>0$ such that
\begin{align*}
	\mathbb{E}\left(\big|X_{n}(\boldsymbol{d})-X_{n-1}(\boldsymbol{d})\big|^{2}\Big|\mathcal{F}_{n-1}\right)\leq{}Cn^{1-\delta}
\end{align*}
for every $n$.

\textbf{Assumption 3.}\ For every $n\geq{}1$ and $\boldsymbol{d}\in(\mathbb{Z}_{0}^{+})^{N}$ we define the sequence $u_{n}(\boldsymbol{d})$ by
\begin{align*}
	p_{\boldsymbol{d}}^{(n)}(\boldsymbol{0})&=1-\frac{u_{n}(\boldsymbol{d})}{n}.
\end{align*}
This is a nonnegative predictable process with respect to the filtration $\mathcal{F}$. We assume that there exists a positive random variable $u(\boldsymbol{d})$ such that $u_{n}(\boldsymbol{d})\to{}u(\boldsymbol{d})$ almost surely as $n\to\infty$.

For every $\boldsymbol{d}\in(\mathbb{Z}_{0}^{+})^{N}$ let us have
\begin{align*}
	H(\boldsymbol{d})=\left\{\boldsymbol{i}=(i_{1},\dots,i_{N})\in(\mathbb{Z}_{0}^{+})^{N}:\sum_{k=1}^{N}i_{k}\geq{}1\right\}.
\end{align*}

\textbf{Assumption 4.}\ For every $\boldsymbol{d}\in(\mathbb{Z}_{0}^{+})^{N}$, where $\sum_{k=1}^{N}d_{k}\geq{}1$, and for every $\boldsymbol{i}\in{}H(\boldsymbol{d})$ there exist nonnegative random variables denoted by $r^{(k)}(\boldsymbol{d}-\boldsymbol{e}_{k})$ such that
\begin{align*}
	\lim_{n\to\infty}np_{\boldsymbol{d}-\boldsymbol{i}}^{(n)}(\boldsymbol{i})&=\left\{
		\begin{array}{l l}
			r^{(k)}(\boldsymbol{d}-\boldsymbol{e}_{k}) & \textrm{if $\boldsymbol{i}=\boldsymbol{e}_{k}$,}\\
			0 & \textmd{if $\sum_{k=1}^{N}i_{k}\geq{}2$}
		\end{array}
	\right.
\end{align*}
holds almost surely.

\textbf{Assumption 5.}\ For every $\boldsymbol{d}\in(\mathbb{Z}_{0}^{+})^{N}$ let $q^{(n)}(\boldsymbol{d})$ denote the conditional probability (with respect to $\mathcal F_{n-1}$) that the new vertex $v_{n}$ attaches to the existing vertices with exactly $d_{k}$ edges of type $k$. We assume that there exists a nonnegative random variable $q(\boldsymbol{d})$ such that $q^{(n)}(\boldsymbol{d})\to{}q(\boldsymbol{d})$ almost surely as $n\to\infty$.

\section{Main results}
\label{sec_MainResults}
\subsection{Asymptotic degree distribution in the general model}
Now we can formulate our general theorem on the asymptotic degree distribution.
\begin{theorem}
	\label{thm_ADD_GM}
	If a random sequence of graphs with multi-type edges satisfies the assumptions above, then for every $\boldsymbol{d}\in(\mathbb{Z}_{0}^{+})^{N}$ we have
	\begin{align*}
		\lim_{n\to\infty}\frac{X_{n}(\boldsymbol{d})}{|V_{n}|}&=x(\boldsymbol{d})\textrm{ a.s.}
	\end{align*}
	The random variables $x(\boldsymbol{d})$ satisfy the following recurrence equation for every $\boldsymbol{d}\in(\mathbb{Z}_{0}^{+})^{N}$:
	\begin{align*}
		x(\boldsymbol{d})&=\frac{1}{u(\boldsymbol{d})+1}\left[\sum_{k=1}^{N}r^{(k)}(\boldsymbol{d}-\boldsymbol{e}_{k})x(\boldsymbol{d}-\boldsymbol{e}_{k})+q(\boldsymbol{d})\right].
	\end{align*}
\end{theorem}

\begin{remark}
	Notice that we have $x(\boldsymbol{d})=0$ if for any $k\in[N]$ we have $d_{k}<0$.
\end{remark}

\subsection{Generalized Barab\'asi--Albert random graph}
\label{model_description_gen_BA}
This is a multi-type version
and a generalization (or modification) of the graph model in \cite{BarabasiAlbert_Albert}, specified in \cite{Bollobas_Riordan_Spencer_Tusnady} (see also \cite{Hofstad,Fazekas_Noszly_Perecsenyi,Ostroumova_Ryabchenko_Samosvat} for general setups). The dynamics of this model is the following: for every $n\geq{}1$, in the $n$th step, the new vertex $v_{n}$ attaches with $M_{n}$ (not necessarily different) edges to some of the old vertices, where $M_{n}$ is a positive integer valued random variable, which is independent of $\mathcal{F}_{n-1}$. The endpoints of the $M_{n}$ edges are chosen independently. The endpoint of each edge is chosen among the existing vertices with probabilities proportional to the degrees. Notice that we do not update degrees until the end of step. The types of the new edges are chosen independently, and the probability of each type is its proportion among the edges of the already existing endpoint of the new edge (not counting the edges added in the actual step).

Now, we list the assumptions on the sequence of random variables $(M_{n})_{n=1}^{\infty}$.

\textbf{Assumption (BA1)} $M_{n}$ is a positive integer valued random variable, which is independent of $\mathcal{F}_{n-1}$ for every $n\geq{}1$.

\textbf{Assumption (BA2)} We assume that there exists a positive random variable $M$ such that $M_{n}\to{}M$ in distribution, and for every $p\geq{}1$ we have $\mathbb{E}(M_{n}^{p})\to\mathbb{E}(M^{p})<\infty$ as $n\to\infty$. The expected value of $M$ will be denoted by $m=\mathbb{E}(M)$.

We need the following lemma to understand the asymptotics of the proportion of edges of type $k$ as the number of steps goes to infinity.

\begin{lemma}
	\label{lemma_AER_GBA}
	For every $k\in[N]$ let us have $\zeta_{n}^{(k)}=\frac{\big|E_{n}^{(k)}\big|}{|E_{n}|}$, i.e.\ the proportion of the number of edges of type $k$ in the generalized Barab\'asi--Albert random graph. For every $k\in[N]$ there exists a random variable $\zeta^{(k)}$ such that $\zeta_{n}^{(k)}\to\zeta^{(k)}$ almost surely as $n\to\infty$.
\end{lemma}

\begin{remark}
	If we have $M_{n}\equiv{}1$ for all $n\geq{}1$, and the initial configuration is a tree, i.e.\ the model is an $N$-type Barab\'asi--Albert random tree, then $\left(\zeta^{(k)},k\in[N]\right)$ has a Dirichlet distribution with parameters $\left(|E_{0}^{(k)}|,k\in[N]\right)$. In this case the number of edges with different types follows a P\'olya urn process.
\end{remark}

\textbf{Asymptotic degree distribution in the generalized Barab\'asi--Albert random graph.}
\label{thm_gen_BA}
\begin{theorem}
	\label{thm_ADD_gen_BA}
	If the assumptions on the sequence $(M_{n})_{n=1}^{\infty}$ are satisfied, then in the generalized Barab\'asi--Albert model for every $\boldsymbol{d}\in(\mathbb{Z}_{0}^{+})^{N}$ we have
	\begin{align*}
		\lim_{n\to\infty}\frac{X_{n}(\boldsymbol{d})}{|V_{n}|}&=x(\boldsymbol{d})\textrm{ a.s.}
	\end{align*}
	The random variables $x(\boldsymbol{d})$ satisfy the following recurrence equation for every $\boldsymbol{d}\in(\mathbb{Z}_{0}^{+})^{N}$:
	\begin{align*}
		x(\boldsymbol{d})&=\sum_{k=1}^{N}\frac{d_{k}-1}{D+2}x(\boldsymbol{d}-\boldsymbol{e}_{k})+\frac{2}{D+2}\mathbb{P}\left(M=D\right)\frac{D!}{\prod_{k=1}^{N}d_{k}!}\prod_{k=1}^{N}\left(\zeta^{(k)}\right)^{d_{k}},
	\end{align*}
	where $\zeta^{(k)}$ is defined in Lemma $\ref{lemma_AER_GBA}$ and $D=\boldsymbol{d}^{T}\boldsymbol{1}=\sum_{k=1}^{N}d_{k}$.
\end{theorem}

\subsection{Model of independent edges}
\label{model_description_MIE}
This model is a modification and a multi-type version of the models in \cite{Dereich_Morters} and \cite{Katona_Mori}, where the new vertex is connected to the old ones independently, with probability depending on the edges of the actual vertex. Instead of connecting with a single edge with a given probability, we add a Poisson number of new edges, with the multiplicative parameter randomly chosen.

In this model, we have the following dynamics: for every $n\geq{}1$, in the $n$th step, the new vertex $v_{n}$ attaches to all of the old vertices with some edges of type $k$ independently. For any existing vertex $w\in{}V_{n-1}$ let $\Delta_{n}^{(k)}(w)$ denote the number edges of type $k$ between the vertices $v_{n}$ and $w$. We assume that, conditionally with respect to $\mathcal{F}_{n-1}$, for every $k\in[N]$ we have
\begin{align*}
	\Delta_{n}^{(k)}(w)\sim\textrm{Poi}\left(\lambda_{n}\frac{\textrm{deg}_{n-1}^{(k)}(v)}{2|E_{n-1}|}\right),
\end{align*}
where $\lambda_{n}$ is a positive random variable. We also assume that for every $w$, the random variables $\left(\Delta_{n}^{(k)}(w)\right)_{k=1}^{N}$ are conditionally independent with respect to $\mathcal{F}_{n-1}$.

Let $\lambda_{1},\lambda_{2},\lambda_{3},\dots$ be a sequence of independent random variables. Similarly to the previous case, we need a few assumptions on their distribution.

\textbf{Assumption (IE1)} For every $n\geq{}1$ the random variable $\lambda_{n}$ is positive and independent of $\mathcal{F}_{n-1}$.

\textbf{Assumption (IE2)} We assume that there exists a positive random variable $\lambda$ such that $\lambda_{n}\to\lambda$ in distribution, and for every $p\geq{}1$ we have $\mathbb{E}(\lambda_{n}^{p})\to\mathbb{E}(\lambda^{p})<\infty$ as $n\to\infty$. The expected value and the variance of $\lambda$ will be denoted by $\mu=\mathbb{E}(\lambda)$ and $\sigma^{2}=\textrm{Var}(\lambda)$, respectively.

For every $n\geq{}1$ we define $\mathcal{F}_{n-1}^{+}=\sigma(\mathcal{F}_{n-1},\lambda_{n})$. Let $\Delta_{n}$ be the number of new edges in the $n$th step, and let $\Delta_{n}^{(k)}$ denote the number of new edges of type $k$ in the $n$th step. For every $n\geq{}1$ we have $\Delta_{n}|\mathcal{F}_{n-1}^{+}\sim\textrm{Poi}(\lambda_{n})$, furthermore for every $k\in[N]$ we have
\begin{align*}
	\Delta_{n}^{(k)}|\mathcal{F}_{n-1}^{+}\sim\textrm{Poi}\left(\lambda_{n}\frac{|E_{n-1}^{(k)}|}{|E_{n-1}|}\right).
\end{align*}
Note that $\left(\Delta_{n}^{(k)}\right)_{k=1}^{N}$ are conditionally independent given $\mathcal{F}_{n-1}^{+}$.

Again, we need the following lemma to understand the asymptotics of the proportion of edges of type $k$ as the number of steps goes to infinity.

\begin{lemma}
	\label{lemma_AER_MIE}
	For every $k\in[N]$ let us have $\hat{\zeta}_{n}^{(k)}=\frac{\big|E_{n}^{(k)}\big|}{|E_{n}|}$, i.e.\ the proportion of the number of edges of type $k$ in the model of independent edges. For every $k\in[N]$ there exists a random variable $\hat{\zeta}^{(k)}$ such that $\hat{\zeta}_{n}^{(k)}\to\hat{\zeta}^{(k)}$ almost surely as $n\to\infty$.
\end{lemma}

\textbf{Asymptotic degree distribution in the model of independent edges.} 
\label{thm_MIE}
\begin{theorem}
	\label{thm_ADD_MIE}
	If the assumptions on the sequence $(\lambda_{n})_{n=1}^{\infty}$ are satisfied, then in the model of independent edges for every $\boldsymbol{d}\in(\mathbb{Z}_{0}^{+})^{N}$ we have
	\begin{align*}
		\lim_{n\to\infty}\frac{X_{n}(\boldsymbol{d})}{|V_{n}|}&=x(\boldsymbol{d})\textrm{ a.s.}
	\end{align*}
	The random variables $x(\boldsymbol{d})$ satisfy the following recurrence equation for every $\boldsymbol{d}\in(\mathbb{Z}_{0}^{+})^{N}$:
	\begin{align*}
		x(\boldsymbol{d})&=\sum_{k=1}^{N}\frac{d_{k}-1}{D+2}x(\boldsymbol{d}-\boldsymbol{e}_{k})+\frac{2}{D+2}\frac{\prod_{k=1}^{N}\left(\hat{\zeta}^{(k)}\right)^{d_{k}}}{\prod_{k=1}^{N}d_{k}!}\mathbb{E}\left(\lambda^{D}e^{-\lambda}\right),
	\end{align*}
	where $\hat{\zeta}^{(k)}$ is defined in Lemma $\ref{lemma_AER_MIE}$ and $D=\boldsymbol{d}^{T}\boldsymbol{1}$.
\end{theorem}

\begin{remark}
	For the calculation of the last term we can use the following. Let us denote by $g_{\lambda}$ the moment generating function of $\lambda$, i.e.\ $g_{\lambda}(t)=\mathbb{E}(e^{t\lambda})$ ($t\in\mathbb{R}$). Let us have $B=\{t\in\mathbb{R}:g_{\lambda}(t)<\infty\}$, i.e.\ the set of finiteness of $g_{\lambda}$, and let $B_{0}$ be the interior of $B$. Suppose that $-1\in{}B_{0}$. It is well known that in this case $g_{\lambda}(t)$ is infinitely differentiable at $t=-1$, furthermore, we have
	\begin{align*}
		g^{\left(D\right)}_{\lambda}(-1)=\mathbb{E}\left(\lambda^{D}e^{-\lambda}\right),
	\end{align*}
	where $D=\boldsymbol{d}^{T}\boldsymbol{1}$ and $g_{\lambda}^{\left(D\right)}$ is the $D$th derivative of $g_{\lambda}$.
\end{remark}

\section{Proofs}
\label{sec_Proofs}
\subsection{The general model}
\begin{definition}
	Two sequences $(a_{n})_{n=1}^{\infty}$ and $(b_{n})_{n=1}^{\infty}$ are asymptotically equal ($a_{n}\sim{}b_{n}$) if they are positive except finitely many terms, and $a_{n}/b_{n}\to{}1$ as $n\to\infty$.
\end{definition}
\begin{definition}
	A sequence $(\beta_{n})_{n=1}^{\infty}$ is regularly varying with exponent $\kappa$ if $\beta_{n}\sim\gamma_{n}n^{\kappa}$,  where $(\gamma_{n})_{n=1}^{\infty}$ is a slowly varying sequence. A sequence $(\gamma_{n})_{n=1}^{\infty}$ is slowly varying if for every positive $s$ we have $\gamma_{[sn]}/\gamma_{n}\to{}1$ as $n\to\infty$.
\end{definition}

We will use the following theorem, see also \cite{Dereich_Ortgiese} for a similar statement.
\begin{lemma}[Lemma 1 in \cite{Backhausz_Mori}]
	\label{lemma_Agi_Tamas}
	Let $\mathcal{F}=(\mathcal{F}_{n})_{n=1}^{\infty}$ be a filtration, $(\xi_{n})_{n=1}^{\infty}$ a nonnegative adapted process with respect to $\mathcal{F}$. Let $(w_{n})_{n=1}^{\infty}$ be a regularly varying sequence of positive numbers with exponent $\kappa>-1$. Suppose that for every $n\geq{}1$,
	\begin{align}
		\label{eq_Agi_Tamas}
		\mathbb{E}\left((\xi_{n}-\xi_{n-1})^{2}\Big|\mathcal{F}_{n-1}\right)&=O\left(n^{1-\delta+2\kappa}\right)
	\end{align}
	holds with some $\delta>0$. Let $(u_{n})_{n=1}^{\infty}$, $(v_{n})_{n=1}^{\infty}$ be nonnegative predictable processes with respect to $\mathcal{F}$ such that $u_{n}<n$ for all $n\geq{}1$.
	
	\begin{enumerate}[(a)]
		\item{}Suppose that
		\begin{align*}
			\mathbb{E}\left(\xi_{n}\big|\mathcal{F}_{n-1}\right)\leq\left(1-\frac{u_{n}}{n}\right)\xi_{n-1}+v_{n},
		\end{align*}
		and $\lim_{n\to\infty}u_{n}=u$, $\limsup_{n\to\infty}v_{n}/w_{n}\leq{}v$ with some random variables $u>0$, $v\geq{}0$. Then we have
		\begin{align*}
			\limsup_{n\to\infty}\frac{\xi_{n}}{nw_{n}}&\leq\frac{v}{u+\kappa+1}\textrm{ a.s.}
		\end{align*}
		\item{}Suppose that
		\begin{align*}
			\mathbb{E}\left(\xi_{n}\big|\mathcal{F}_{n-1}\right)\geq\left(1-\frac{u_{n}}{n}\right)\xi_{n-1}+v_{n},
		\end{align*}
		and $\lim_{n\to\infty}u_{n}=u$, $\liminf_{n\to\infty}v_{n}/w_{n}\geq{}v$ with some random variables $u>0$, $v\geq{}0$. Then we have
		\begin{align*}
			\liminf_{n\to\infty}\frac{\xi_{n}}{nw_{n}}&\geq\frac{v}{u+\kappa+1}\textrm{ a.s.}
		\end{align*}
	\end{enumerate}
\end{lemma}
We will use this lemma for the sequence $w_{n}\equiv{}1$ and $\kappa=0$.

\subsection*{Proof of Theorem \ref{thm_ADD_GM}} We prove the theorem by induction on $\boldsymbol{d}^{T}\boldsymbol{1}$. If $\boldsymbol{d}^{T}\boldsymbol{1}$ is negative, then the proof is trivial. Let $\boldsymbol{d}\in(\mathbb{Z}_{0}^{+})^{N}$ be a fixed vector, such that $\boldsymbol{d}^{T}\boldsymbol{1}\geq{}0$. Notice that, for every $n\geq{}1$, in the $n$th step, the value of $X_{n}(\boldsymbol{d})$ may change due to the following events:
\begin{itemize}
	\item{}an existing vertex with generalized degree $\boldsymbol{d}$ is connected to the new vertex;
	\item{}an existing vertex with generalized degree $\boldsymbol{d}-\boldsymbol{i}=\left(d_{k}-i_{k}\right)_{k=1}^{N}$ is chosen, and it gets $i_{k}$ new edges of type $k$;
	\item{}the new vertex attaches to the old vertices with $d_{k}$ edges of type $k$ for every $k\in[N]$.
\end{itemize}
For every $n\geq{}1$, in the $n$th step, we have
\begin{align}
	\label{eq_GM}
	&\mathbb{E}\left[X_{n}(\boldsymbol{d})\big|\mathcal{F}_{n-1}\right]=X_{n-1}(\boldsymbol{d})p_{\boldsymbol{d}}^{(n)}(\boldsymbol{0})+\left[\sum_{\boldsymbol{i}\in{}H(\boldsymbol{d})}X_{n-1}(\boldsymbol{d}-\boldsymbol{i})p_{\boldsymbol{d}-\boldsymbol{i}}^{(n)}(\boldsymbol{i})\right]+q^{(n)}(\boldsymbol{d}),
\end{align}
where 
\begin{align*}
	H(\boldsymbol{d})=\left\{\boldsymbol{i}=(i_{1},\dots,i_{N})\in(\mathbb{Z}_{0}^{+})^{N}:\sum_{k=1}^{N}i_{k}\geq{}1\right\}.
\end{align*}
Assumption 2 implies that there exists a positive $\delta$ and a positive $C$ such that for every $n\geq{}1$ we have
\begin{align*}
	\mathbb{E}\left(\big|X_{n}(\boldsymbol{d})-X_{n-1}(\boldsymbol{d})\big|^{2}\Big|\mathcal{F}_{n-1}\right)\leq{}Cn^{1-\delta}.
\end{align*}
With this $\delta$, equation \eqref{eq_Agi_Tamas} in Lemma \ref{lemma_Agi_Tamas} is satisfied with $\xi_{n}=X_{n}(\boldsymbol{d})$. We want to rewrite equation \eqref{eq_GM} in the following form:
\begin{align*}
	\mathbb{E}\left[X_{n}(\boldsymbol{d})\big|\mathcal{F}_{n-1}\right]&=X_{n-1}(\boldsymbol{d})\left[1-\frac{u_{n}(\boldsymbol{d})}{n}\right]+v_{n}(\boldsymbol{d}),
\end{align*}
where the processes $\left(u_{n}(\boldsymbol{d})\right)_{n=1}^{\infty}$ and $\left(v_{n}(\boldsymbol{d})\right)_{n=1}^{\infty}$ satisfy the assumptions of Lemma \ref{lemma_Agi_Tamas}. Recall the definition of $u_{n}(\boldsymbol{d})$ from Assumption 3. It is easy to see that this process is predictable with respect to $\mathcal{F}$. Assumption 3 implies that there exists a positive random variable $u(\boldsymbol{d})$ such that $u_{n}(\boldsymbol{d})\to{}u(\boldsymbol{d})$ almost surely as $n\to\infty$. We define $H'(\boldsymbol{d})=H(\boldsymbol{d})\setminus\left\{\boldsymbol{e}_{k},k\in[N]\right\}$.

We define
\begin{align*}
	v_{n}(\boldsymbol{d})&=\sum_{k=1}^{N}X_{n-1}(\boldsymbol{d}-\boldsymbol{e}_{k})p_{\boldsymbol{d}-\boldsymbol{e}_{k}}^{(n)}(\boldsymbol{e}_{k})+\left[\sum_{\boldsymbol{i}\in{}H'(\boldsymbol{d})}X_{n-1}(\boldsymbol{d}-\boldsymbol{i})p_{\boldsymbol{d}-\boldsymbol{i}}^{(n)}(\boldsymbol{i})\right]+q^{(n)}(\boldsymbol{d}).
\end{align*}
It is easy to see that this process is predictable with respect to $\mathcal{F}$. Using Assumptions 4 and 5 and the induction hypothesis, we conclude that there exists a nonnegative random variable $v(\boldsymbol{d})$, such that
\begin{align*}
	v_{n}(\boldsymbol{d})\to{}v(\boldsymbol{d})=\sum_{k=1}^{N}r^{(k)}(\boldsymbol{d}-\boldsymbol{e}_{k})x(\boldsymbol{d}-\boldsymbol{e}_{k})+q(\boldsymbol{d})\textrm{ a.s.}
\end{align*}
as $n\to\infty$. Lemma \ref{lemma_Agi_Tamas} implies that
\begin{align*}
	\lim_{n\to\infty}\frac{X_{n}(\boldsymbol{d})}{n}=\frac{v(\boldsymbol{d})}{u(\boldsymbol{d})+1}\textrm{ a.s.}
\end{align*}
Since $|V_{n}|\sim{}n$, the proof of Theorem \ref{thm_ADD_GM} is complete. \hfill$\Box$

\subsection{Generalized Barab\'asi--Albert random graph}
First, for every $n\geq{}0$ we define $\mathcal{F}_{n}^{+}=\sigma(\mathcal{F}_{n},M_{n+1})$. We show that $|E_{n}|\sim{}mn$, where $m=\mathbb{E}(M)$. For every $n\geq{}1$ we have $|E_{n}|=\sum_{k=1}^{N}\big|E_{0}^{(k)}\big|+\sum_{i=1}^{n}M_{i}$. By the assumptions of the model, the sequence $(M_{n})_{n=1}^{\infty}$ satisfies the following conditions:
\begin{align*}
	\lim_{n\to\infty}\frac{1}{n}\sum_{i=1}^{n}\mathbb{E}(M_{i})=\mathbb{E}(M)=m>0\quad\textrm{and}\quad\sum_{n=1}^{\infty}\frac{\textrm{Var}(M_{n})}{n^{2}}<\infty.
\end{align*}
Therefore Kolmogorov's theorem can be applied (Theorem 6.7.\ in \cite{Petrov}) for the sequence $(M_{n})_{n=1}^{\infty}$, thus we have $|E_{n}|\sim{}mn$.

We will use the following lemma, which can be proved by Bonferroni's inequality.
\begin{lemma}
	\label{lemma_Approximation}
	For every $n\geq{}1$ and $x\in[0,1]$ we have
	\begin{align*}
		\left|(1-x)^{n}-(1-nx)\right|&\leq{n \choose 2}x^{2}.
	\end{align*}
\end{lemma}

\subsection*{Proof of Lemma \ref{lemma_AER_GBA}}
First, let us fix $k\in[N]$. For every $n\geq{}1$ the distribution of the number of new edges of type $k$ in the $n$th step conditionally with respect to $\mathcal{F}_{n-1}^{+}$ is $\textrm{Bin}\left(M_{n},\frac{\big|E_{n-1}^{(k)}\big|}{|E_{n-1}|}\right)$. For every $n\geq{}1$ we have
\begin{align*}
	\mathbb{E}\left(\frac{\big|E_{n}^{(k)}\big|}{|E_{n}|}\Bigg|\mathcal{F}_{n-1}^{+}\right)&=\frac{\big|E_{n-1}^{(k)}\big|}{|E_{n-1}|+M_{n}}+\frac{M_{n}\frac{\big|E_{n-1}^{(k)}\big|}{|E_{n-1}|}}{|E_{n-1}|+M_{n}}=\frac{\big|E_{n-1}^{(k)}\big|\left(1+\frac{M_{n}}{|E_{n-1}|}\right)}{|E_{n-1}|+M_{n}}=\frac{\big|E_{n-1}^{(k)}\big|}{|E_{n-1}|}.
\end{align*}
This is $\mathcal{F}_{n-1}$-measurable, hence this yields
\begin{align*}
	\mathbb{E}\left(\frac{\big|E_{n}^{(k)}\big|}{|E_{n}|}\Bigg|\mathcal{F}_{n-1}\right)&=\frac{\big|E_{n-1}^{(k)}\big|}{|E_{n-1}|}.
\end{align*}
We conclude that $\left(\zeta_{n}^{(k)},\mathcal{F}_{n}\right)_{n=1}^{\infty}$ is a nonnegative martingale, thus it is convergent almost surely. Let $\zeta^{(k)}\geq{}0$ be its limit. The proof of Lemma \ref{lemma_AER_GBA} is complete. \hfill$\Box$

\subsection*{Proof of Theorem \ref{thm_ADD_gen_BA}} We use Theorem \ref{thm_ADD_GM}, so we have to check the assumptions of the general model.

\textbf{Assumption 1.}\ By the dynamics of the model, it is easy to see that Assumption 1 trivially holds.

\textbf{Assumption 2.}\ Assumption (BA2) implies that, for every $n\geq{}1$ and $\boldsymbol{d}\in(\mathbb{Z}_{0}^{+})^{N}$ we have
\begin{align*}
\mathbb{E}\left(\big|X_{n}(\boldsymbol{d})-X_{n-1}(\boldsymbol{d})\big|^{2}\Big|\mathcal{F}_{n-1}\right)&\leq\mathbb{E}(M_{n}^{2})\to\mathbb{E}(M^{2})<\infty
\end{align*}
as $n\to\infty$. If we choose $\delta=1$, then Assumption 2 is satisfied.

\textbf{Assumption 3.}\ For every $n\geq{}1$ and $\boldsymbol{d}\in(\mathbb{Z}_{0}^{+})^{N}$ we have
\begin{align*}
	p_{\boldsymbol{d}}^{(n)}(\boldsymbol{0})&=\mathbb{E}\left[\left(1-\frac{\boldsymbol{d}^{T}\boldsymbol{1}}{2|E_{n-1}|}\right)^{M_{n}}\Bigg|\mathcal{F}_{n-1}\right].
\end{align*}
To calculate the expected value above, we will use the following formula:
\begin{align*}
	&\mathbb{E}\left[\left(1-\frac{\boldsymbol{d}^{T}\boldsymbol{1}}{2|E_{n-1}|}\right)^{M_{n}}\Bigg|\mathcal{F}_{n-1}\right]=\mathbb{E}\left(1-M_{n}\frac{\boldsymbol{d}^{T}\boldsymbol{1}}{2|E_{n-1}|}\Bigg|\mathcal{F}_{n-1}\right)+\eta_{n}(\boldsymbol{d}),
\end{align*}
where
\begin{align*}
	\eta_{n}(\boldsymbol{d})&=\mathbb{E}\left[\left(1-\frac{\boldsymbol{d}^{T}\boldsymbol{1}}{2|E_{n-1}|}\right)^{M_{n}}-\left(1-M_{n}\frac{\boldsymbol{d}^{T}\boldsymbol{1}}{2|E_{n-1}|}\right)\Bigg|\mathcal{F}_{n-1}\right].
\end{align*}
Lemma \ref{lemma_Approximation} implies that for every $n\geq\boldsymbol{d}^{T}\boldsymbol{1}$ we have
\begin{align*}
	&\left|\left(1-\frac{\boldsymbol{d}^{T}\boldsymbol{1}}{2|E_{n-1}|}\right)^{M_{n}}-\left(1-M_{n}\frac{\boldsymbol{d}^{T}\boldsymbol{1}}{2|E_{n-1}|}\right)\right|\leq{M_{n} \choose 2}\left(\frac{\boldsymbol{d}^{T}\boldsymbol{1}}{2|E_{n-1}|}\right)^{2}.
\end{align*}
By using the above bound, we obtain that
\begin{align*}
	|\eta_{n}(\boldsymbol{d})|&\leq\mathbb{E}\left[\left|\left(1-\frac{\boldsymbol{d}^{T}\boldsymbol{1}}{2|E_{n-1}|}\right)^{M_{n}}-\left(1-M_{n}\frac{\boldsymbol{d}^{T}\boldsymbol{1}}{2|E_{n-1}|}\right)\right|\Bigg|\mathcal{F}_{n-1}\right]\\
	&\leq{}\mathbb{E}\left[{M_{n} \choose 2}\left(\frac{\boldsymbol{d}^{T}\boldsymbol{1}}{2|E_{n-1}|}\right)^{2}\Bigg|\mathcal{F}_{n-1}\right]=\left(\frac{\boldsymbol{d}^{T}\boldsymbol{1}}{2|E_{n-1}|}\right)^{2}\mathbb{E}\left[{M_{n} \choose 2}\right]\\
	&\leq{}\left(\frac{\boldsymbol{d}^{T}\boldsymbol{1}}{2|E_{n-1}|}\right)^{2}\mathbb{E}(M_{n}^{2})
\end{align*}
almost surely, by $|E_{n-1}|\sim{}mn$, and Assumption (BA2). The definition of $u_{n}(\boldsymbol{d})$ and $\eta_{n}(\boldsymbol{d})$ implies that
\begin{align*}
	u_{n}(\boldsymbol{d})&=n\left(1-\left[\mathbb{E}\left(1-M_{n}\frac{\boldsymbol{d}^{T}\boldsymbol{1}}{2|E_{n-1}|}\Bigg|\mathcal{F}_{n-1}\right)+\eta_{n}(\boldsymbol{d})\right]\right)\\
	&=n\frac{\boldsymbol{d}^{T}\boldsymbol{1}}{2}\cdot\frac{\mathbb{E}(M_{n})}{|E_{n-1}|}-n\cdot\eta_{n}(\boldsymbol{d}).
\end{align*}
This is $\mathcal{F}_{n-1}$-measurable, hence $(u_{n}(\boldsymbol{d}))_{n=1}^{\infty}$ is a predictable process with respect to $\mathcal{F}$. Recall that $|E_{n-1}|\sim{}mn$ and $n\cdot|\eta_{n}(\boldsymbol{d})|=o(1)$ almost surely. Assumption (BA2) implies that
\begin{align*}
	u(\boldsymbol{d})&=\lim_{n\to\infty}u_{n}(\boldsymbol{d})=\frac{\boldsymbol{d}^{T}\boldsymbol{1}}{2}\textrm{ a.s.}
\end{align*}

\textbf{Assumption 4.}\ First, we fix $k\in[N]$. For every $n\geq{}1$ and $\boldsymbol{d}\in(\mathbb{Z}_{0}^{+})^{N}$, where $\boldsymbol{d}^{T}\boldsymbol{1}\geq{}1$, the following holds:
\begin{align}
\label{eq_Assumption4_GBA}
	p_{\boldsymbol{d}-\boldsymbol{e}_{k}}^{(n)}(\boldsymbol{e}_{k})&=\mathbb{E}\left[M_{n}\left(\frac{d_{k}-1}{2|E_{n-1}|}\right)\left(1-\frac{\boldsymbol{d}^{T}\boldsymbol{1}-1}{2|E_{n-1}|}\right)^{M_{n}-1}\Bigg|\mathcal{F}_{n-1}\right]\\
	&=\left(\frac{d_{k}-1}{2|E_{n-1}|}\right)\mathbb{E}\left[M_{n}\left(1-\frac{\boldsymbol{d}^{T}\boldsymbol{1}-1}{2|E_{n-1}|}\right)^{M_{n}-1}\Bigg|\mathcal{F}_{n-1}\right]\nonumber.
\end{align}
Similarly to the previous case, we obtain that
\begin{align*}
	&\mathbb{E}\left[M_{n}\left(1-\frac{\boldsymbol{d}^{T}\boldsymbol{1}-1}{2|E_{n-1}|}\right)^{M_{n}-1}\Bigg|\mathcal{F}_{n-1}\right]=\mathbb{E}\left[M_{n}\left(1-(M_{n}-1)\frac{\boldsymbol{d}^{T}\boldsymbol{1}-1}{2|E_{n-1}|}\right)\Bigg|\mathcal{F}_{n-1}\right]+\eta_{n}(\boldsymbol{d}),
\end{align*}
where
\begin{align*}
	\eta_{n}(\boldsymbol{d})&=\mathbb{E}\left[M_{n}\left(1-\frac{\boldsymbol{d}^{T}\boldsymbol{1}-1}{2|E_{n-1}|}\right)^{M_{n}-1}-M_{n}\left(1-(M_{n}-1)\frac{\boldsymbol{d}^{T}\boldsymbol{1}-1}{2|E_{n-1}|}\right)\Bigg|\mathcal{F}_{n-1}\right],
\end{align*}
which is not the same sequence as the $\eta$'s from the previous section. Lemma \ref{lemma_Approximation} implies that for every $n\geq\boldsymbol{d}^{T}\boldsymbol{1}$ we have
\begin{align*}
	&\left|M_{n}\left(1-\frac{\boldsymbol{d}^{T}\boldsymbol{1}-1}{2|E_{n-1}|}\right)^{M_{n}-1}-M_{n}\left(1-(M_{n}-1)\frac{\boldsymbol{d}^{T}\boldsymbol{1}-1}{2|E_{n-1}|}\right)\right|\leq{}M_{n}{M_{n}-1 \choose 2}\left(\frac{\boldsymbol{d}^{T}\boldsymbol{1}-1}{2|E_{n-1}|}\right)^{2}.
\end{align*}
Combining this with Assumption (BA2), we obtain that
\begin{align*}
	|\eta_{n}(\boldsymbol{d})|&\leq\mathbb{E}\left[\left|M_{n}\left(1-\frac{\boldsymbol{d}^{T}\boldsymbol{1}-1}{2|E_{n-1}|}\right)^{M_{n}-1}-M_{n}\left(1-(M_{n}-1)\frac{\boldsymbol{d}^{T}\boldsymbol{1}-1}{2|E_{n-1}|}\right)\right|\Bigg|\mathcal{F}_{n-1}\right]\\
	&\leq{}\mathbb{E}\left[M_{n}{M_{n}-1 \choose 2}\left(\frac{\boldsymbol{d}^{T}\boldsymbol{1}-1}{2|E_{n-1}|}\right)^{2}\Bigg|\mathcal{F}_{n-1}\right]=\left(\frac{\boldsymbol{d}^{T}\boldsymbol{1}-1}{2|E_{n-1}|}\right)^{2}\mathbb{E}\left[M_{n}{M_{n}-1 \choose 2}\right]\\
	&\leq{}\left(\frac{\boldsymbol{d}^{T}\boldsymbol{1}-1}{2|E_{n-1}|}\right)^{2}\mathbb{E}(M_{n}^{3})=o\left(\frac{1}{n}\right)\textrm{ a.s.}
\end{align*}
Getting back to equation \eqref{eq_Assumption4_GBA}, we conclude that
\begin{align*}
	p_{\boldsymbol{d}-\boldsymbol{e}_{k}}^{(n)}(\boldsymbol{e}_{k})&=\left(\frac{d_{k}-1}{2|E_{n-1}|}\right)\mathbb{E}\left[M_{n}\left(1-\frac{\boldsymbol{d}^{T}\boldsymbol{1}-1}{2|E_{n-1}|}\right)^{M_{n}-1}\Bigg|\mathcal{F}_{n-1}\right]\\
	&=\left(\frac{d_{k}-1}{2|E_{n-1}|}\right)\left(\mathbb{E}\left[M_{n}\left(1-(M_{n}-1)\frac{\boldsymbol{d}^{T}\boldsymbol{1}-1}{2|E_{n-1}|}\right)\Bigg|\mathcal{F}_{n-1}\right]+\eta_{n}(\boldsymbol{d})\right)\\
	&=\left(\frac{d_{k}-1}{2|E_{n-1}|}\right)\left(\mathbb{E}(M_{n})-\mathbb{E}\left[M_{n}(M_{n}-1)\frac{\boldsymbol{d}^{T}\boldsymbol{1}-1}{2|E_{n-1}|}\Bigg|\mathcal{F}_{n-1}\right]+\eta_{n}(\boldsymbol{d})\right)\\
	&\sim\frac{d_{k}-1}{2}\cdot\frac{1}{n}+o\left(\frac{1}{n}\right)\textrm{ a.s.}
\end{align*}
Therefore, for every $k\in[N]$ we have
\begin{align*}
	\lim_{n\to\infty}np_{\boldsymbol{d}-\boldsymbol{e}_{k}}^{(n)}(\boldsymbol{e}_{k})&=r^{(k)}(\boldsymbol{d}-\boldsymbol{e}_{k})=\frac{d_{k}-1}{2}\textrm{ a.s.}
\end{align*}

Let $\boldsymbol{i}\in{}H'(\boldsymbol{d})$, i.e.\ $\forall{}k\in[N]:0\leq{}i_{k}\leq{}d_{k}$ and $\boldsymbol{i}^{T}\boldsymbol{1}\geq{}2$. In this case, we can bound the conditional expectation as follows:
\begin{align*}
	&p_{\boldsymbol{d}-\boldsymbol{i}}^{(n)}(\boldsymbol{i})\\
	&=\mathbb{E}\left[\frac{M_{n}!}{\left(\prod_{k=1}^{N}i_{k}!\right)\left(M_{n}-\boldsymbol{i}^{T}\boldsymbol{1}\right)!}\left[\prod_{k=1}^{N}\left(\frac{d_{k}-i_{k}}{2|E_{n-1}|}\right)^{i_{k}}\right]\left(1-\frac{(\boldsymbol{d-i})^{T}\boldsymbol{1}}{2|E_{n-1}|}\right)^{M_{n}-\boldsymbol{i}^{T}\boldsymbol{1}}\Bigg|\mathcal{F}_{n-1}\right]\\
	&\leq{}\prod_{k=1}^{N}\left(\frac{d_{k}-i_{k}}{2|E_{n-1}|}\right)^{i_{k}}\mathbb{E}\left[\frac{M_{n}!}{\left(\prod_{k=1}^{N}i_{k}!\right)\left(M_{n}-\boldsymbol{i}^{T}\boldsymbol{1}\right)!}\right]\leq\frac{\prod_{k=1}^{N}(d_{k}-i_{k})^{i_{k}}}{\left(2|E_{n-1}|\right)^{\boldsymbol{i}^{T}\boldsymbol{1}}}\mathbb{E}\left(M_{n}^{\boldsymbol{i}^{T}\boldsymbol{1}}\right).
\end{align*}
This yields
\begin{align*}
	\lim_{n\to\infty}np_{\boldsymbol{d}-\boldsymbol{i}}^{(n)}(\boldsymbol{i})=0\textrm{ a.s.},
\end{align*}
due to Assumption (BA2) and the fact that $|E_{n-1}|\sim{}mn$.

\textbf{Assumption 5.}\ By the dynamics of the model, we conclude that for every $n\geq{}1$ and $\boldsymbol{d}\in(\mathbb{Z}_{0}^{+})^{N}$ the following holds:
\begin{align*}
	q^{(n)}(\boldsymbol{d})&=\mathbb{E}\left[I\left(M_{n}=\boldsymbol{d}^{T}\boldsymbol{1}\right)\frac{\left(\boldsymbol{d}^{T}\boldsymbol{1}\right)!}{\prod_{k=1}^{N}d_{k}!}\prod_{k=1}^{N}\left(\frac{\big|E_{n-1}^{(k)}\big|}{|E_{n-1}|}\right)^{d_{k}}\Bigg|\mathcal{F}_{n-1}\right]\\
	&=\mathbb{P}\left(M_{n}=\boldsymbol{d}^{T}\boldsymbol{1}\right)\frac{\left(\boldsymbol{d}^{T}\boldsymbol{1}\right)!}{\prod_{k=1}^{N}d_{k}!}\prod_{k=1}^{N}\left(\frac{\big|E_{n-1}^{(k)}\big|}{|E_{n-1}|}\right)^{d_{k}}.
\end{align*}
Assumption (BA1) implies that $\mathbb{P}\left(M_{n}=\boldsymbol{d}^{T}\boldsymbol{1}\right)\to\mathbb{P}\left(M=\boldsymbol{d}^{T}\boldsymbol{1}\right)$ as $n\to\infty$. It follows from Lemma \ref{lemma_AER_GBA} that
\begin{align*}
	q(\boldsymbol{d})&=\lim_{n\to\infty}q^{(n)}(\boldsymbol{d})=\mathbb{P}\left(M=\boldsymbol{d}^{T}\boldsymbol{1}\right)\frac{\left(\boldsymbol{d}^{T}\boldsymbol{1}\right)!}{\prod_{k=1}^{N}d_{k}!}\prod_{k=1}^{N}\left(\zeta^{(k)}\right)^{d_{k}}\textrm{ a.s.}
\end{align*}
This yields
\begin{align*}
	u(\boldsymbol{d})&=\frac{\boldsymbol{d}^{T}\boldsymbol{1}}{2},\\
	r^{(k)}(\boldsymbol{d}-\boldsymbol{e}_{k})&=\frac{d_{k}-1}{2}\quad(\forall{}k\in[N])\\
	q(\boldsymbol{d})&=\mathbb{P}\left(M=\boldsymbol{d}^{T}\boldsymbol{1}\right)\frac{\left(\boldsymbol{d}^{T}\boldsymbol{1}\right)!}{\prod_{k=1}^{N}d_{k}!}\prod_{k=1}^{N}\left(\zeta^{(k)}\right)^{d_{k}}.
\end{align*}

Applying Theorem \ref{thm_ADD_GM} we get Theorem \ref{thm_ADD_gen_BA}. \hfill$\Box$

\subsection{Model of independent edges} We will use the following lemma.
\begin{lemma}
	For the number of edges, we have the following asymptotics: $|E_{n}|\sim\mu{}n$.
\end{lemma}
\begin{proof}
	Let us have $\Delta_{0}=|E_{0}|$ and $\lambda_{0}=0$. We define 
	
	\begin{align*}
		Z_{n}=\sum_{i=0}^{n}\Delta_{i}-\lambda_{i}=|E_{n}|-\sum_{i=1}^{n}\lambda_{i}.
	\end{align*}
	
	We show that $(Z_{n},\mathcal{F}_{n})_{n=1}^{\infty}$ is a square integrable martingale, i.e.\ $(Z_{n},\mathcal{F}_{n})_{n=1}^{\infty}$ is a martingale, and we have $\mathbb{E}(Z_{n}^{2})<\infty$ for every $n\geq{}1$.
	
	For every $n\geq{}1$ we have
	\begin{align*}
		\mathbb{E}(Z_{n}|\mathcal{F}_{n-1})&=\mathbb{E}(Z_{n-1}+\Delta_{n}-\lambda_{n}|\mathcal{F}_{n-1})=\\
		&=Z_{n-1}+\mathbb{E}\left[\mathbb{E}(\Delta_{n}|\mathcal{F}_{n-1}^{+})-\lambda_{n}|\mathcal{F}_{n-1}\right]=Z_{n-1},
	\end{align*}
	since $\Delta_{n}|\mathcal{F}_{n-1}^{+}\sim\textrm{Poi}(\lambda_{n})$.
	
	Furthermore, we can bound the expectation of the squares as follows:
	\begin{align*}
		\mathbb{E}(Z_{n}^{2})&=\mathbb{E}\left[\left(\sum_{i=1}^{n}\Delta_{i}-\lambda_{i}\right)^{2}\right]=\mathbb{E}\left[\sum_{i=1}^{n}(\Delta_{i}-\lambda_{i})^{2}+2\sum_{i<j}(\Delta_{i}-\lambda_{i})(\Delta_{j}-\lambda_{j})\right]\\
		&=\sum_{i=1}^{n}\mathbb{E}\left[(\Delta_{i}-\lambda_{i})^{2}\right]+2\sum_{i<j}\mathbb{E}\left[(\Delta_{i}-\lambda_{i})(\Delta_{j}-\lambda_{j})\right]\\
		&=\sum_{i=1}^{n}\mathbb{E}\left(\mathbb{E}\left[(\Delta_{i}-\lambda_{i})^{2}\Big|\mathcal{F}_{i-1}^{+}\right]\right)+2\sum_{i<j}\mathbb{E}\left(\mathbb{E}\left[(\Delta_{i}-\lambda_{i})(\Delta_{j}-\lambda_{j})\Big|\mathcal{F}_{j-1}^{+}\right]\right)\\
		&=\sum_{i=1}^{n}\mathbb{E}(\lambda_{i})<\infty,
	\end{align*}
	hence $(Z_{n},\mathcal{F}_{n})_{n=1}^{\infty}$ is a square integrable martingale. The increasing process associated with $Z_{n}^{2}$ by the Doob decomposition is the following:
	\begin{align*}
		A_{n}&=\sum_{i=1}^{n}\textrm{Var}\left(\Delta_{i}|\mathcal{F}_{i-1}\right)=\sum_{i=1}^{n}\mathbb{E}(\Delta_{i}^{2}|\mathcal{F}_{i-1})-\mathbb{E}^{2}(\Delta_{i}|\mathcal{F}_{i-1})\\
		&=\sum_{i=1}^{n}\mathbb{E}\left[\mathbb{E}(\Delta_{i}^{2}|\mathcal{F}_{i-1}^{+})\Big|\mathcal{F}_{i-1}\right]-\mathbb{E}^{2}\left[\mathbb{E}(\Delta_{i}|\mathcal{F}_{i-1}^{+})\Big|\mathcal{F}_{i-1}\right]\\
		&=\sum_{i=1}^{n}\mathbb{E}(\lambda_{i}^{2}+\lambda_{i})-\mathbb{E}^{2}(\lambda_{i})\\
		&=\sum_{i=1}^{n}\textrm{Var}(\lambda_{i})+\mathbb{E}(\lambda_{i})\leq{}n(\mu+\sigma^{2}).
	\end{align*}
	By using \cite{Neveu}, Proposition VII-2-4, we conclude that $|E_{n}|=\left(\sum_{i=1}^{n}\lambda_{i}\right)n+o\left(n^{1/2+\varepsilon}\right)$ almost surely as $n\to\infty$ on the event $\{A_{n}\to\infty\}$ for all $\varepsilon>0$.
	
	For every $n\geq{}1$ we have $|E_{n}|=\sum_{k=1}^{N}\big|E_{0}^{(k)}\big|+\sum_{i=1}^{n}\Delta_{i}$. By the assumptions of the model, the sequence $(\lambda_{i})_{i=1}^{n}$ satisfies the following conditions:
	\begin{align*}
		\lim_{n\to\infty}\frac{1}{n}\sum_{i=1}^{n}\mathbb{E}(\lambda_{i})=\mathbb{E}(\lambda)=\mu\quad\textrm{and}\quad\sum_{n=1}^{\infty}\frac{\textrm{Var}(\lambda_{n})}{n^{2}}<\infty.
	\end{align*}
	Therefore, Kolmogorov's theorem can be applied (Theorem 6.7.\ in \cite{Petrov}) for the sequence $(\lambda_{n})_{n=1}^{\infty}$. We get that $|E_{n}|\sim\mu{}n$.
\end{proof}

\subsection*{Proof of Lemma \ref{lemma_AER_MIE}.}\ Recall that for a fix $k\in[N]$ we have
\begin{align*}
	\Delta_{n}^{(k)}\sim\textrm{Poi}\left(\lambda_{n}\frac{|E_{n}^{(k)}|}{|E_{n}|}\right)\qquad\textrm{and}\qquad\Delta_{n}-\Delta_{n}^{(k)}\sim\textrm{Poi}\left(\lambda_{n}\left[1-\frac{|E_{n}^{(k)}|}{|E_{n}|}\right]\right),
\end{align*}
furthermore $\Delta_{n}^{(k)}$ and $\Delta_{n}-\Delta_{n}^{(k)}$ are conditionally independent given $\mathcal{F}_{n-1}$. Because of this, it is enough to prove this lemma for $N=2$, which means there are only two types.

We are going to show that we have  $\Delta_{n}^{(1)}\Big|\mathcal{F}_{n-1}^{+}\sim\textrm{Bin}\left(\Delta_{n},\frac{|E_{n-1}^{(1)}|}{|E_{n-1}|}\right)$. For every $n\geq{}1$ we define $\mathcal{F}_{n-1}^{++}=\sigma(\mathcal{F}_{n-1}^{+},\Delta_{n})$. For all $i\leq{}j$ the conditional distribution can be calculated as follows:
\begin{align*}
	&\mathbb{P}\left(\Delta_{n}^{(1)}=i\Big|\Delta_{n}=j,\mathcal{F}_{n-1}^{+}\right)=\frac{\mathbb{P}\left(\Delta_{n}^{(1)}=i,\Delta_{n}=j\Big|\mathcal{F}_{n-1}^{+}\right)}{\mathbb{P}\left(\Delta_{n}^{(1)}+\Delta_{n}^{(2)}=j\Big|\mathcal{F}_{n-1}^{+}\right)}\\
	&=\frac{\mathbb{P}\left(\Delta_{n}^{(1)}=i,\Delta_{n}^{(2)}=j-i\Big|\mathcal{F}_{n-1}^{+}\right)}{\mathbb{P}\left(\Delta_{n}^{(1)}+\Delta_{n}^{(2)}=j\Big|\mathcal{F}_{n-1}^{+}\right)}=\frac{\mathbb{P}\left(\Delta_{n}^{(1)}=i\Big|\mathcal{F}_{n-1}^{+}\right)\cdot\mathbb{P}\left(\Delta_{n}^{(2)}=j-i\Big|\mathcal{F}_{n-1}^{+}\right)}{\mathbb{P}\left(\Delta_{n}^{(1)}+\Delta_{n}^{(2)}=j\Big|\mathcal{F}_{n-1}^{+}\right)}\\
	&=\frac{\frac{\left(\lambda_{n}\frac{|E_{n-1}^{(1)}|}{|E_{n-1}|}\right)^{i}}{i!}\cdot\exp\left(-\lambda_{n}\frac{|E_{n-1}^{(1)}|}{|E_{n-1}|}\right)\cdot\frac{\left(\lambda_{n}\frac{|E_{n-1}^{(2)}|}{|E_{n-1}|}\right)^{j-i}}{(j-i)!}\cdot\exp\left(-\lambda_{n}\frac{|E_{n-1}^{(2)}|}{|E_{n-1}|}\right)}{\frac{\lambda_{n}^{j}}{j!}\cdot\exp\left(-\lambda_{n}\right)}\\
	&={j \choose i}\left(\frac{|E_{n-1}^{(1)}|}{|E_{n-1}|}\right)^{i}\left(1-\frac{|E_{n-1}^{(1)}|}{|E_{n-1}|}\right)^{j-i}.
\end{align*}
For all $n\geq{}1$, similarly to the proof of Lemma \ref{lemma_AER_GBA}, we have
\begin{align*}
	\mathbb{E}\left(\frac{|E_{n}^{(1)}|}{|E_{n}|}\Bigg|\mathcal{F}_{n-1}^{++}\right)&=\frac{|E_{n-1}^{(1)}|}{|E_{n-1}|+\Delta_{n}}+\frac{\Delta_{n}\frac{|E_{n-1}^{(1)}|}{|E_{n-1}|}}{|E_{n-1}|+\Delta_{n}}=\frac{|E_{n-1}^{(1)}|}{|E_{n-1}|}.
\end{align*}
Notice that $E_{n-1}$ and $E_{n-1}^{(1)}$ are $\mathcal{F}_{n-1}$-measurable, which implies that
\begin{align*}
\mathbb{E}\left(\frac{|E_{n}^{(1)}|}{|E_{n}|}\Bigg|\mathcal{F}_{n-1}\right)&=\frac{|E_{n-1}^{(1)}|}{|E_{n-1}|}.
\end{align*}
We conclude that $\left(\hat{\zeta}_{n}^{(1)},\mathcal{F}_{n}\right)_{n=1}^{\infty}$ is a nonnegative martingale, thus it is convergent almost surely. Let $\hat{\zeta}^{(1)}\geq{}0$ be its limit. The proof of Lemma \ref{lemma_AER_MIE} is complete. \hfill$\Box$

\subsection*{Proof of Theorem \ref{thm_ADD_MIE}.} We will use Theorem \ref{thm_ADD_GM}, so we have to check the assumptions of the general model.

\textbf{Assumption 1.}\ Again, Assumption 1 trivially holds.

\textbf{Assumption 2.}\ By using $\Delta_{n}|\mathcal{F}_{n-1}^{+}\sim\textrm{Poi}(\lambda_{n})$ we obtain that for all $\boldsymbol{d}\in(\mathbb{Z}_{0}^{+})^{N}$ we have
\begin{align*}
	&\mathbb{E}\left[\big|X_{n}(\boldsymbol{d})-X_{n-1}(\boldsymbol{d})\big|^{2}\Big|\mathcal{F}_{n-1}\right]\leq\mathbb{E}(\Delta_{n}^{2}|\mathcal{F}_{n-1})=\mathbb{E}\left[\mathbb{E}(\Delta_{n}^{2}|\mathcal{F}_{n-1}^{+})\Big|\mathcal{F}_{n-1}\right]\\
	&=\mathbb{E}(\lambda_{n}^{2}+\lambda_{n}|\mathcal{F}_{n-1})=\mathbb{E}(\lambda_{n}^{2}+\lambda_{n})\to\sigma^{2}+\mu^{2}+\mu<\infty
\end{align*}
as $n\to\infty$. If we choose $\delta=1$, then Assumption 2 is satisfied.

\textbf{Assumption 3.}\ For every $n\geq{}1$ and $\boldsymbol{d}\in(\mathbb{Z}_{0}^{+})^{N}$ we have
\begin{align*}
	p_{\boldsymbol{d}}^{(n)}(\boldsymbol{0})&=\mathbb{E}\left[\exp\left(-\lambda_{n}\frac{\boldsymbol{d}^{T}\boldsymbol{1}}{2|E_{n-1}|}\right)\Bigg|\mathcal{F}_{n-1}\right].
\end{align*}

We will use Taylor expansion. In order to do this, we write the expectation in the following form:
\begin{align*}
	&\mathbb{E}\left[\exp\left(-\lambda_{n}\frac{\boldsymbol{d}^{T}\boldsymbol{1}}{2|E_{n-1}|}\right)\Bigg|\mathcal{F}_{n-1}\right]=\mathbb{E}\left(1-\lambda_{n}\frac{\boldsymbol{d}^{T}\boldsymbol{1}}{2|E_{n-1}|}\Bigg|\mathcal{F}_{n-1}\right)+\eta_{n}(\boldsymbol{d}),
\end{align*}
where
\begin{align*}
	\eta_{n}(\boldsymbol{d})&=\mathbb{E}\left[\exp\left(-\lambda_{n}\frac{\boldsymbol{d}^{T}\boldsymbol{1}}{2|E_{n-1}|}\right)-\left(1-\lambda_{n}\frac{\boldsymbol{d}^{T}\boldsymbol{1}}{2|E_{n-1}|}\right)\Bigg|\mathcal{F}_{n-1}\right].
\end{align*}
It is well known that for all $x\geq{}0$ we have $|e^{-x}-(1-x)|\leq\frac{x^{2}}{2}$, which implies that
\begin{align*}
	\left|\exp\left(-\lambda_{n}\frac{\boldsymbol{d}^{T}\boldsymbol{1}}{2|E_{n-1}|}\right)-\left(1-\lambda_{n}\frac{\boldsymbol{d}^{T}\boldsymbol{1}}{2|E_{n-1}|}\right)\right|\leq\frac{1}{2}\left(\lambda_{n}\frac{\boldsymbol{d}^{T}\boldsymbol{1}}{2|E_{n-1}|}\right)^{2}.
\end{align*}

By using the above inequality, we obtain that
\begin{align*}
	|\eta_{n}(\boldsymbol{d})|\leq{}&\mathbb{E}\left[\left|\exp\left(-\lambda_{n}\frac{\boldsymbol{d}^{T}\boldsymbol{1}}{2|E_{n-1}|}\right)-\left(1-\lambda_{n}\frac{\boldsymbol{d}^{T}\boldsymbol{1}}{2|E_{n-1}|}\right)\right|\Bigg|\mathcal{F}_{n-1}\right]\\
	\leq{}&\mathbb{E}\left[\frac{1}{2}\left(\lambda_{n}\frac{\boldsymbol{d}^{T}\boldsymbol{1}}{2|E_{n-1}|}\right)^{2}\Bigg|\mathcal{F}_{n-1}\right]=\mathbb{E}\left[\lambda_{n}^{2}\frac{\left(\boldsymbol{d}^{T}\boldsymbol{1}\right)^{2}}{8|E_{n-1}|^{2}}\Bigg|\mathcal{F}_{n-1}\right]\\
	=&\mathbb{E}(\lambda_{n}^{2})\frac{\left(\boldsymbol{d}^{T}\boldsymbol{1}\right)^{2}}{8|E_{n-1}|^{2}}=o\left(\frac{1}{n}\right)\textrm{ a.s.}
\end{align*}
by Assumption (IE2) and $|E_{n-1}|\sim\mu{}n$. The definition of $u_{n}(\boldsymbol{d})$ and $\eta_{n}(\boldsymbol{d})$ implies that
\begin{align*}
	u_{n}(\boldsymbol{d})&=n\left(\frac{\boldsymbol{d}^{T}\boldsymbol{1}}{2}\cdot\frac{\mathbb{E}(\lambda_{n})}{|E_{n-1}|}-\eta_{n}(\boldsymbol{d})\right).
\end{align*}
This is $\mathcal{F}_{n-1}$-measurable, hence $(u_{n}(\boldsymbol{d}))_{n=1}^{\infty}$ is a predictable process with respect to $\mathcal{F}$. Recall that $|E_{n-1}|\sim\mu{}n$, and $n\cdot\eta_{n}(\boldsymbol{d})=o(1)$ almost surely. Assumption (IE2) implies that
\begin{align*}
	u(\boldsymbol{d})&=\lim_{n\to\infty}u_{n}(\boldsymbol{d})=\frac{\boldsymbol{d}^{T}\boldsymbol{1}}{2}\textrm{ a.s.}
\end{align*}

\textbf{Assumption 4.}\ First, we fix $k\in[N]$. For every $n\geq{}1$ and $\boldsymbol{d}\in(\mathbb{Z}_{0}^{+})^{N}$, where $\boldsymbol{d}^{T}\boldsymbol{1}\geq{}1$, we have
\begin{align}
\label{eq_Assumption4_MIE}
	p_{\boldsymbol{d}-\boldsymbol{e}_{k}}^{(n)}(\boldsymbol{e}_{k})&=\mathbb{E}\left[\lambda_{n}\frac{d_{k}-1}{2|E_{n-1}|}\cdot\exp\left(-\lambda_{n}\frac{d_{k}-1}{2|E_{n-1}|}\right)\Bigg|\mathcal{F}_{n-1}\right]\\
	&=\frac{d_{k}-1}{2|E_{n-1}|}\mathbb{E}\left[\lambda_{n}\cdot\exp\left(-\lambda_{n}\frac{d_{k}-1}{2|E_{n-1}|}\right)\Bigg|\mathcal{F}_{n-1}\right]\nonumber.
\end{align}
Similarly to the previous case, we obtain that
\begin{align*}
	&\mathbb{E}\left[\lambda_{n}\cdot\exp\left(-\lambda_{n}\frac{d_{k}-1}{2|E_{n-1}|}\right)\Bigg|\mathcal{F}_{n-1}\right]=\mathbb{E}\left[\lambda_{n}\left(1-\lambda_{n}\frac{d_{k}-1}{2|E_{n-1}|}\right)\Bigg|\mathcal{F}_{n-1}\right]+\eta_{n}(\boldsymbol{d}),
\end{align*}
where
\begin{align*}
	\eta_{n}(\boldsymbol{d})&=\mathbb{E}\left[\lambda_{n}\cdot\exp\left(-\lambda_{n}\frac{d_{k}-1}{2|E_{n-1}|}\right)-\lambda_{n}\left(1-\lambda_{n}\frac{d_{k}-1}{2|E_{n-1}|}\right)\Bigg|\mathcal{F}_{n-1}\right].
\end{align*}
Again, by using $|e^{-x}-(1-x)|\leq\frac{x^{2}}{2}$ for all $x\geq{}0$, we conclude that
\begin{align*}
	\left|\lambda_{n}\cdot\exp\left(-\lambda_{n}\frac{d_{k}-1}{2|E_{n-1}|}\right)-\lambda_{n}\left(1-\lambda_{n}\frac{d_{k}-1}{2|E_{n-1}|}\right)\right|\leq\frac{\lambda_{n}}{2}\left(\lambda_{n}\frac{d_{k}-1}{2|E_{n-1}|}\right)^{2}.
\end{align*}
Combining this with Assumption (IE2), we obtain that
\begin{align*}
	|\eta_{n}(\boldsymbol{d})|&\leq\mathbb{E}\left[\left|\lambda_{n}\cdot\exp\left(-\lambda_{n}\frac{d_{k}-1}{2|E_{n-1}|}\right)-\lambda_{n}\left(1-\lambda_{n}\frac{d_{k}-1}{2|E_{n-1}|}\right)\right|\Bigg|\mathcal{F}_{n-1}\right]\\
	&\leq\mathbb{E}\left[\frac{\lambda_{n}}{2}\left(\lambda_{n}\frac{d_{k}-1}{2|E_{n-1}|}\right)^{2}\Bigg|\mathcal{F}_{n-1}\right]=\mathbb{E}\left[\lambda_{n}^{3}\frac{(d_{k}-1)^{2}}{8|E_{n-1}|^{2}}\Bigg|\mathcal{F}_{n-1}\right]\\
	&=\mathbb{E}(\lambda_{n}^{3})\frac{(d_{k}-1)^{2}}{8|E_{n-1}|^{2}}=o\left(\frac{1}{n}\right)\textrm{ a.s.}
\end{align*}
By using this we conclude that
\begin{align*}
	p_{\boldsymbol{d}-\boldsymbol{e}_{k}}^{(n)}(\boldsymbol{e}_{k})&=\frac{d_{k}-1}{2|E_{n-1}|}\mathbb{E}\left[\lambda_{n}\cdot\exp\left(-\lambda_{n}\frac{d_{k}-1}{2|E_{n-1}|}\right)\Bigg|\mathcal{F}_{n-1}\right]\\
	&=\frac{d_{k}-1}{2|E_{n-1}|}\left(\mathbb{E}\left[\lambda_{n}\left(1-\lambda_{n}\frac{d_{k}-1}{2|E_{n-1}|}\right)\Bigg|\mathcal{F}_{n-1}\right]+\eta_{n}(\boldsymbol{d})\right)\\
	&\sim\frac{d_{k}-1}{2}\cdot\frac{1}{n}+o\left(\frac{1}{n}\right)\textrm{ a.s.}
\end{align*}
Putting this together, we obtain that for every $k\in[N]$ we have
\begin{align*}
	\lim_{n\to\infty}np_{\boldsymbol{d}-\boldsymbol{e}_{k}}^{(n)}(\boldsymbol{e}_{k})&=r^{(k)}(\boldsymbol{d}-\boldsymbol{e}_{k})=\frac{d_{k}-1}{2}\textrm{ a.s.}
\end{align*}

Now let $\boldsymbol{i}\in{}H'(\boldsymbol{d})$, i.e.\ $\forall{}k\in[N]:0\leq{}i_{k}\leq{}d_{k}$ and $\boldsymbol{i}^{T}\boldsymbol{1}\geq{}2$. For every $n\geq{}1$ we have
\begin{align*}
	p_{\boldsymbol{d}-\boldsymbol{i}}^{(n)}(\boldsymbol{i})&=\mathbb{E}\left[\prod_{k=1}^{N}\frac{1}{i_{k}!}\left(\lambda_{n}\frac{d_{k}-i_{k}}{2|E_{n-1}|}\right)^{i_{k}}\exp\left(-\lambda_{n}\frac{d_{k}-i_{k}}{2|E_{n-1}|}\right)\Bigg|\mathcal{F}_{n-1}\right]\\
	&=\frac{\prod_{k=1}^{N}(d_{k}-i_{k})^{i_{k}}}{\left(2|E_{n-1}|\right)^{\boldsymbol{i}^{T}\boldsymbol{1}}\prod_{k=1}^{N}i_{k}!}\mathbb{E}\left[\lambda_{n}^{\boldsymbol{i}^{T}\boldsymbol{1}}\cdot\exp\left(-\lambda_{n}\frac{(\boldsymbol{d-i})^{T}\boldsymbol{1}}{2|E_{n-1}|}\right)\Bigg|\mathcal{F}_{n-1}\right]\\
	&\leq\frac{\prod_{k=1}^{N}(d_{k}-i_{k})^{i_{k}}}{\left(2|E_{n-1}|\right)^{\boldsymbol{i}^{T}\boldsymbol{1}}\prod_{k=1}^{N}i_{k}!}\mathbb{E}\left(\lambda_{n}^{\boldsymbol{i}^{T}\boldsymbol{1}}\right),
\end{align*}
which implies that
\begin{align*}
	\lim_{n\to\infty}np_{\boldsymbol{d}-\boldsymbol{i}}^{(n)}(\boldsymbol{i})=0\textrm{ a.s.}
\end{align*}

\textbf{Assumption 5.}\ By the dynamics of the model, for every $n\geq{}1$ and $\boldsymbol{d}\in(\mathbb{Z}_{0}^{+})^{N}$, the following holds:
\begin{align*}
	q^{(n)}(\boldsymbol{d})&=\mathbb{E}\left[\mathbb{P}\left(\bigcap_{k=1}^{N}\left\{\Delta_{n}^{(k)}=d_{k}\right\}\Big|\mathcal{F}_{n-1}^{+}\right)\Bigg|\mathcal{F}_{n-1}\right]\\
	&=\mathbb{E}\left[\prod_{k=1}^{N}\mathbb{P}\left(\Delta_{n}^{(k)}=d_{k}\Big|\mathcal{F}_{n-1}^{+}\right)\Bigg|\mathcal{F}_{n-1}\right]\\
	&=\mathbb{E}\left[\prod_{k=1}^{N}\frac{1}{d_{k}!}\left(\lambda_{n}\frac{\big|E_{n-1}^{(k)}\big|}{|E_{n-1}|}\right)^{d_{k}}\exp\left(-\lambda_{n}\frac{\big|E_{n-1}^{(k)}\big|}{|E_{n-1}|}\right)\Bigg|\mathcal{F}_{n-1}\right]\\
	&=\frac{1}{\prod_{k=1}^{N}d_{k}!}\prod_{k=1}^{N}\left(\frac{\big|E_{n-1}^{(k)}\big|}{|E_{n-1}|}\right)^{d_{k}}\mathbb{E}\left(\lambda_{n}^{\boldsymbol{d}^{T}\boldsymbol{1}}\cdot\exp\left(-\lambda_{n}\right)\Big|\mathcal{F}_{n-1}\right).
\end{align*}

By Lemma \ref{lemma_AER_MIE} and the independence of $\lambda_{n}$ and $\mathcal{F}_{n-1}$, we have
\begin{align*}
	q(\boldsymbol{d})&=\lim_{n\to\infty}q^{(n)}(\boldsymbol{d})=\lim_{n\to\infty}\frac{\prod_{k=1}^{N}\left(\hat{\zeta}_{n-1}^{(k)}\right)^{d_{k}}}{\prod_{k=1}^{N}d_{k}!}\mathbb{E}\left(\lambda_{n}^{\boldsymbol{d}^{T}\boldsymbol{1}}e^{-\lambda_{n}}\right)\\
	&=\frac{\prod_{k=1}^{N}\left(\hat{\zeta}^{(k)}\right)^{d_{k}}}{\prod_{k=1}^{N}d_{k}!}\mathbb{E}\left(\lambda^{\boldsymbol{d}^{T}\boldsymbol{1}}e^{-\lambda}\right)\textrm{ a.s.},
\end{align*}
since the function $t^{\boldsymbol{d}^{T}\boldsymbol{1}}e^{-t}$ is bounded and continuous and $\lambda_n\rightarrow\lambda$ in distribution.

We obtain that
\begin{align*}
	u(\boldsymbol{d})&=\frac{\boldsymbol{d}^{T}\boldsymbol{1}}{2},\\
	r^{(k)}(\boldsymbol{d}-\boldsymbol{e}_{k})&=\frac{d_{k}-1}{2}\quad(\forall{}k\in[N])\\
	q(\boldsymbol{d})&=\frac{\prod_{k=1}^{N}\left(\hat{\zeta}^{(k)}\right)^{d_{k}}}{\prod_{k=1}^{N}d_{k}!}\mathbb{E}\left(\lambda^{\boldsymbol{d}^{T}\boldsymbol{1}}e^{-\lambda}\right).
\end{align*}

Applying Theorem \ref{thm_ADD_GM} we get Theorem \ref{thm_ADD_MIE}. \hfill$\Box$

\section{Scale-free property of random graphs in the multi-type case}
A scale-free graph model is a random graph whose degree distribution follows a power law, i.e.\ the proportion of vertices with degree $d$ asymptotically equals to $d^{-\gamma}$, where $\gamma>0$ is a deterministic constant. It is well known that many large real networks have this property, see e.g.\ \cite{Hofstad}, although there are discussions about how common they are \cite{Broido_Clauset}.

The formal definition of scale-free property of random graphs with no types is the following.
\begin{definition}
We assume that the proportion of vertices with degree $d$ converges to a deterministic constant $c_{d}$ a.s. for all $d\geq{}0$, and the sum of the sequence $(c_{d})_{d=0}^{\infty}$ equals to 1. In this case the sequence $(c_{d})_{d=0}^{\infty}$ is an asymptotic degree distribution. Furthermore, if $c_{d}d^{\gamma}\to{}C$ as $d\to\infty$ holds with some positive $C$, then the model has the scale-free property, and $\gamma$ is the so-called characteristic exponent.
\end{definition}

We are going to use the following theorem.

\begin{theoremx}[Theorem 1 in \cite{Backhausz_Mori_2}]
	\label{thm_recurrence_equation}
	Consider the following recurrence equation:
	\begin{align*}
		x_{n}&=\sum_{j=1}^{n-1}w_{n,j}x_{n-j}+r_{n},\qquad{}w_{n,j}=a_{j}+\frac{b_{j}}{n}+c_{n,j}, \qquad{}(n=1,2,3,\dots),
	\end{align*}
	where $w_{n,j}\geq{}0$, and $a_{n}$, $b_{n}$, $c_{n,j}$, $r_{n}$ satisfy the following conditions.
	\begin{description}
		\item[(r1)]$a_{n}\geq{}0$ for $n\geq{}1$, and the greatest common divisor of the set
		$\{n:a_{n}>0\}$ is $\mathrm{1}$;
		\item[(r2)]$r_{n}\geq{}0$, and there exists such an $n$ that $r_{n}>0$;
		\item[(r3)]there exists $z>0$ such that
		\begin{align*}
			1<\sum_{n=1}^{\infty}a_{n}z^{n}&<\infty,\qquad\sum_{n=1}^{\infty}|b_{n}|z^{n}<\infty,\\
			\sum_{n=1}^{\infty}\sum_{j=1}^{\infty}|c_{n,j}|z^{j}&<\infty,\qquad\sum_{n=1}^{\infty}r_{n}z^{n}<\infty.
		\end{align*}
	\end{description}
	Suppose that the sequence $(x_{n})_{n=1}^{\infty}$ satisfies the recurrence equation, conditions $(r1)$-$(r3)$ hold, and $(x_{n})_{n=1}^{\infty}$ has infinitely many positive terms. Then $x_{n}n^{-\gamma}q^{n}\to{}C$ as $n\to\infty$, where $C$ is a positive constant, $q$ is the positive solution of equation $\sum_{n=1}^{\infty}a_{n}q^{n}=1$, and
	\begin{align*}
		\gamma&=\frac{\sum_{n=1}^{\infty}b_{n}q^{n}}{\sum_{n=1}^{\infty}na_{n}q^{n}}.
	\end{align*}
\end{theoremx}

\subsection{Scale-free property of the generalized Barab\'asi--Albert random graph}
In addition to the assumptions in Section \ref{model_description_gen_BA}, we also assume that $M_{1},M_{2},M_{3},\dots$ is a sequence of identically distributed random variables and there exists $z>1$ such that $\sum_{l=1}^{\infty}\frac{\mathbb{E}(M_{1}^{l})}{l\cdot{}l!}z^{l}<\infty$. The last assumption is trivially fulfilled if $\sup_{l}\mathbb{E}(M_{1}^{l})<\infty$.

First, let us fix $k\in[N]$. For all $l\geq{}0$ we define $X_{n}^{(k)}(l)=\bigg|\left\{v\in{}V_{n}:\textrm{deg}_{n}^{(k)}(v)=l\right\}\bigg|$, i.e.\ the number of vertices in $G_{n}$ with $l$ edges of type $k$ connected to them. The asymptotic degree distribution of type $k$ edges is $\left(x^{(k)}_{l}\right)_{l=0}^{\infty}$, where $x^{(k)}_{l}$ is defined as the almost sure limit of the sequence $\left(\frac{X_{n}^{(k)}(l)}{|V_{n}|}\right)_{n=0}^{\infty}$ as $n\to\infty$. Recall that in every step the endpoints of the new edges are chosen independently of each other and the degrees of the existing vertices are not updated until the end of the step. By using this, we conclude that for every $l\geq{}0$ the change in the value of $X_{n}^{(k)}(l)$ only depends on the edges of type $k$, thus we have
\begin{align}
\label{eq_scale_free_BA}
	&\mathbb{E}\left[X_{n}^{(k)}(l)|\mathcal{F}_{n-1}\right]=X_{n-1}^{(k)}(l)\mathbb{E}\left[\left(1-\frac{l}{2|E_{n-1}|}\right)^{M_{n}}\Bigg|\mathcal{F}_{n-1}\right]\\
	&+\sum_{i=1}^{l-1}X_{n-1}^{(k)}(l-i)\cdot\mathbb{E}\left[{M_{n} \choose i}\left(\frac{l-i}{2|E_{n-1}|}\right)^{i}\left(1-\frac{l-i}{2|E_{n-1}|}\right)^{M_{n}-i}\Bigg|\mathcal{F}_{n-1}\right]\nonumber\\
	&+\mathbb{E}\left[{M_{n} \choose k}\left(\zeta_{n-1}^{(k)}\right)^{l}\left(1-\zeta_{n-1}^{(k)}\right)^{M_{n}-l}\Bigg|\mathcal{F}_{n-1}\right],\nonumber
\end{align}
where $\zeta_{n-1}^{(k)}$ is the proportion of edges of type $k$ in $G_{n-1}$. By using Lemma \ref{lemma_Agi_Tamas} and the same arguments as in the proof of Theorem \ref{thm_ADD_gen_BA}, we can show that $x^{(k)}_{l}$ exists for all $l$ and we can find the recurrence equations for the asymptotic degree distribution. The only part which is different to the previous sections is finding the almost sure limit of the last term in equation \eqref{eq_scale_free_BA} as the number of steps goes to infinity. Since $M_{n}$ is independent of $\mathcal{F}_{n-1}$ and $\zeta_{n-1}^{(k)}$ is measurable with respect to $\mathcal{F}_{n-1}$, we have
\begin{align}
	\label{eq_mean}
	\mathbb{E}\left[{M_{n} \choose l}\left(\zeta_{n-1}^{(k)}\right)^{l}\left(1-\zeta_{n-1}^{(k)}\right)^{M_{n}-l}\Big|\mathcal{F}_{n-1}\right]&=\mathbb{E}\left[{M_{n} \choose l}t^{l}(1-t)^{M_{n}-l}\right]\Bigg|_{t=\zeta_{n-1}^{(k)}}.
\end{align}
Recall that $(M_{n})_{n=1}^{\infty}$ is a sequence of identically distributed random variables, thus we define $f(t)=\mathbb{E}\left[{M_{1} \choose l}t^{l}(1-t)^{M_{1}-l}\right]$, where $t\in[0,1]$. By using Weierstrass' M-test, we are going to show that $f$ is continuous. For all $t\in[0,1]$, we have
\begin{align*}
	f(t)&=\mathbb{E}\left[{M_{1} \choose l}t^{l}(1-t)^{M_{1}-l}\right]=\sum_{i=l}^{\infty}{i \choose l}t^{l}(1-t)^{i-l}\mathbb{P}(M_{1}=i)\\
	&\leq\sum_{i=l}^{\infty}{i \choose l}\mathbb{P}(M_{1}=i)=\mathbb{E}\left[{M_{1} \choose l}\right]\leq\mathbb{E}\left(M_{1}^{l}\right)<\infty,
\end{align*}
by the Assumption (BA2), thus $f$ is continuous. Since $\zeta_{n}^{(k)}\to\zeta^{(k)}$ almost surely as $n\to\infty$ and $f$ is continuous, equation \eqref{eq_mean} implies that
\begin{align*}
	\lim_{n\to\infty}\mathbb{E}\left[{M_{n} \choose l}\left(\zeta_{n-1}^{(k)}\right)^{l}\left(1-\zeta_{n-1}^{(k)}\right)^{M_{n}-l}\Big|\mathcal{F}_{n-1}\right]&=\mathbb{E}\left[{M_{1} \choose l}\left(\zeta^{(k)}\right)^{l}\left(1-\zeta^{(k)}\right)^{M_{1}-l}\right] \mathrm{a.s.}
\end{align*}

By using Lemma \ref{lemma_Agi_Tamas}, for every $l$ we have
\begin{align}
\label{eq_recurrence_equation_gen_BA}
	x^{(k)}_{l}&=\frac{l-1}{l+2}x^{(k)}_{l-1}+\frac{2}{l+2}\mathbb{E}\left[{M_{1} \choose l}\left(\zeta^{(k)}\right)^{l}\left(1-\zeta^{(k)}\right)^{M_{1}-l}\right].
\end{align}

We are going to apply Theorem \ref{thm_recurrence_equation}. The last equation can be written as
\begin{align*}
	x^{(k)}_{l}&=\left[1-\frac{3}{l}+\left(\frac{3}{l}-\frac{3}{l+2}\right)\right]x^{(k)}_{l-1}+\frac{2}{l+2}\mathbb{E}\left[{M_{1} \choose l}\left(\zeta^{(k)}\right)^{l}\left(1-\zeta^{(k)}\right)^{M_{1}-l}\right].
\end{align*}
If we choose
\begin{align*}
	a_{1}=1,\qquad{}a_{j}=0\quad{}\textrm{for}\quad{}j\geq{}2,\\
	b_{1}=-3,\qquad{}b_{j}=0\quad\textrm{for}\quad{}j\geq{}2,\\
	c_{l,1}=\frac{3}{l}-\frac{3}{l+2},\qquad{}c_{l,j}=0\quad{}\textrm{for}\quad{}j\geq{}2,\\
	r_{l}^{(k)}=\frac{2}{l+2}\mathbb{E}\left[{M_{1} \choose l}\left(\zeta^{(k)}\right)^{l}\left(1-\zeta^{(k)}\right)^{M_{1}-l}\right],
\end{align*}
then the assumptions $\bf{(r1)}$ and $\bf{(r3)}$ of Theorem \ref{thm_recurrence_equation} are fulfilled by also using the fact that there exists $z>1$ such that $\sum_{l=1}^{\infty}\frac{\mathbb{E}(M_{1}^{l})}{l\cdot{}l!}z^{l}<\infty$. We know that there exists $l>0$ such that $\mathbb{P}(M_{1}=l)>0$. By using Lemma \ref{lemma_AER_GBA}, we conclude that $\zeta^{(k)}|M_{1}=l$ is positive with positive probability thus $r_{l}^{(k)}>0$ and the assumption $\bf{(r2)}$ is satisfied.

\begin{remark}
	If in addition to the assumptions in Section \ref{model_description_gen_BA}, we also assume that $M_{i}\equiv{}M$ for all $i\geq{}1$ where $M$ is a positive integer then in this case the proportion of edges of type $k$ has an absolutely continuous almost sure limit (see e.g.\ Theorem 3 in \cite{Chen_Kuba}), thus none of the types die out asymptotically with probability one.
\end{remark}

By using Theorem \ref{thm_recurrence_equation}, we conclude that for every $k\in[N]$ we have
\begin{align*}
	x^{(k)}_{l}l^{3}\to{}C_{k}
\end{align*}
as $l\to\infty$ for some positive $C_{k}$, thus the characteristic exponent equals to $3$.

Finally, for all $d\geq{}0$ we define $Z_{n}(d)=\bigg|\left\{v\in{}V_{n}:\sum_{k=1}^{N}\textrm{deg}_{n}^{(k)}(v)=d\right\}\bigg|$, i.e.\ the number of vertices in $G_{n}$ with $d$ edges connected to them. This way we get back to the single-type graph models. The asymptotic degree distribution is $(z_{d})_{d=0}^{\infty}$, where $z_{d}$ is defined as the almost sure limit of the sequence $\left(\frac{Z_{n}(d)}{|V_{n}|}\right)_{n=0}^{\infty}$ as $n\to\infty$. For every $d\geq{}0$ we have
\begin{align*}
	&\mathbb{E}\left[Z_{n}(d)|\mathcal{F}_{n-1}\right]=Z_{n-1}(d)\mathbb{E}\left[\left(1-\frac{d}{2|E_{n-1}|}\right)^{M_{n}}\Bigg|\mathcal{F}_{n-1}\right]\\
	&+\sum_{i=1}^{d-1}Z_{n-1}(k-i)\mathbb{E}\left[{M_{n} \choose i}\left(\frac{d-i}{2|E_{n-1}|}\right)^{i}\left(1-\frac{d-i}{2|E_{n-1}|}\right)^{M_{n}-i}\Bigg|\mathcal{F}_{n-1}\right]\nonumber\\
	&+\mathbb{P}\left(M_{n}=d|\mathcal{F}_{n-1}\right).
\end{align*}
By using the same argument as in the previous section, we have
\begin{align*}
	z_{d}d^{3}\to{}C
\end{align*}
as $d\to\infty$ for some positive $C$, thus the characteristic exponent equals to $3$. This
provides a generalization on some of the results of preferential attachment
models (see e.g.\ \cite{Hofstad}). As the calculation above shows, this model fits into the general framework of \cite{Fazekas_Noszly_Perecsenyi} or \cite{Ostroumova_Ryabchenko_Samosvat} for single-type preferential attachment random graphs.

\subsection{Scale-free property of the model of independent edges}
In the model of independent edges we can use the same arguments. In addition to the assumptions in Section \ref{model_description_MIE}, we also assume that $\lambda_{1},\lambda_{2},\lambda_{3},\dots$ is a sequence of identically distributed random variables and there exists $z>1$ such that $\sum_{l=1}^{\infty}\frac{\mathbb{E}(\lambda_{1}^{l})}{l\cdot{}l!}z^{l}<\infty$.

In the model of independent edges for every type $k$ we have $r^{(k)}_{l}=\frac{2}{l+2}\mathbb{E}\left(\frac{\left(\lambda_{1}\hat{\zeta}^{(k)}\right)^{l}}{l!}e^{-\lambda_{1}\hat{\zeta}^{(k)}}\right)$, where $\hat{\zeta}^{(k)}$ is the asymptotic proportion of edges of type $k$. Similarly to the previous subsection, by using Lemma \ref{lemma_AER_MIE}, we know that $\left(\hat{\zeta}^{(k)},\mathcal{F}_{n}\right)_{n=1}^{\infty}$ is a martingale and for every $k\in[N]$ we have $\big|E_{0}^{(k)}\big|>0$, thus $\hat{\zeta}^{(k)}$ is positive with positive probability and the last assumption of Theorem \ref{thm_recurrence_equation} is fulfilled.

In this special case we can prove the same results as in the previous subsection. For every $l\geq{}0$ we define $\hat{X}_{n}^{(k)}(l)=\left|\left\{v\in{}V_{n}:\textrm{deg}_{n}^{(k)}(v)=l\right\}\right|$. The asymptotic degree distribution of type $k$ edges is $\left(\hat{x}_{l}^{(k)}\right)_{l=0}^{\infty}$, where $\hat{x}_{l}^{(k)}$ is defined as the almost sure limit of the sequence $\left(\frac{\hat{X}_{n}^{(k)}(l)}{|V_{n}|}\right)_{n=0}^{\infty}$ as $n\to\infty$. For every $k\in[N]$ we have
\begin{align*}
	\hat{x}_{l}^{(k)}\to\hat{C}_{k}
\end{align*}
as $l\to\infty$ for some positive $\hat{C}_{k}$, and the characteristic exponent equals to $3$. Again, for every $d\geq{}0$ we define $\hat{Z}_{n}(d)=\left|\left\{v\in{}V_{n}:\sum_{k=1}^{N}\textrm{deg}_{n}^{(k)}(v)=d\right\}\right|$. The asymptotic degree distribution is $(\hat{z}_{d})_{d=0}^{\infty}$, where $\hat{z}_{d}$ is defined as the almost sure limit of the sequence $\left(\frac{\hat{Z}_{n}(d)}{|V_{n}|}\right)_{n=0}^{\infty}$ as $n\to\infty$. By using the same argument as in the previous subsection, we have
\begin{align*}
	\hat{z}_{d}d^{3}\to{}\hat{C}
\end{align*}
as $d\to\infty$ for some positive $\hat{C}$, thus the characteristic exponent equals to $3$.

\section*{Acknowledgements}
This research was partially supported by Pallas Athene Domus Educationis Foundation. The views expressed are those of the authors’ and do not necessarily reflect the official opinion of Pallas Athene Domus Educationis Foundation. The first author was supported by the Bolyai Research Grant of the Hungarian Academy of Sciences.

The authors thank the referees for their review and appreciate the comments and
suggestions which contributed to improve the quality of the article.

\end{document}